\documentclass{amsart}
\usepackage{amsmath}
\usepackage{amsthm}
\usepackage{amssymb}
\usepackage{enumitem}

\setenumerate{wide}

\usepackage{color}
\usepackage{ifthen}
\usepackage{scalerel}
\usepackage{mathrsfs}
\usepackage[mathscr]{euscript}
\usepackage{lineno}

\newtheorem{theorem}{Theorem}[section]
\newtheorem{proposition}[theorem]{Proposition}
\newtheorem{lemma}[theorem]{Lemma}
\newtheorem{corollary}[theorem]{Corollary}

\theoremstyle{definition}

\newtheorem{remark}[theorem]{Remark}

\newcommand{\inter}{\operatorname{int}}

\newcommand{\klim}{\operatorname{K-}\lim}

\newcommand{\norm}[1]{\left\lVert #1 \right\rVert}

\newcommand{\floor}[1]{\left\lfloor #1 \right\rfloor}
\newcommand{\bra}[1]{\left( #1 \right)}

\renewcommand{\tilde}{\widetilde}
\newcommand{\abs}[1]{\left|#1\right|}
\newcommand{\set}[2]{\left\{ #1 \ \middle| \ #2 \right\} }
\newcommand{\ceil}[1]{\left\lceil #1 \right\rceil}

\newcommand{\e}{\varepsilon}

\newcommand{\NN}{\mathbb{N}}

\newcommand{\ZZ}{\mathbb{Z}}
\newcommand{\RR}{\mathbb{R}}

\newcommand{\UU}{\mathbb{U}}

\newcommand{\cN}{\mathcal{N}}
\newcommand{\cB}{\mathscr{B}}

\newcommand{\Ball}[3]{\cB_{#1}\bra{#2,#3} }
\newcommand{\bBall}[3]{\bar\cB_{#1}\bra{#2,#3} }
\newcommand{\BallR}[3]{\cB_\infty^{#1}\bra{#2,#3} }
\newcommand{\bBallR}[3]{\bar\cB_\infty^{#1}\bra{#2,#3} }

\newcommand{\QM}[2]{\mathcal{QM}_{#1}\ifthenelse{\equal{#2}{}}{}{^{\leq #2}}}
\newcommand{\QA}[2]{\mathcal{QA}_{#1}\ifthenelse{\equal{#2}{}}{}{^{\leq #2}}}

\newcommand{\SM}[2]{\mathcal{SM}_{#1}\ifthenelse{\equal{#2}{}}{}{^{\leq #2}}}
\newcommand{\SA}[2]{\mathcal{SA}_{#1}\ifthenelse{\equal{#2}{}}{}{^{\leq #2}}}

\newcommand{\QMD}[2]{{\mathcal{QM}}_{#1}\ifthenelse{\equal{#2}{}}{}{^{\leq #2}}(\UU)}
\newcommand{\SMD}[2]{{\mathcal{SM}}_{#1}\ifthenelse{\equal{#2}{}}{}{^{\leq #2}}(\UU)}

\renewcommand{\subset}{\subseteq}
\renewcommand{\supset}{\supseteq}

\newcommand*\patchAmsMathEnvironmentForLineno[1]{\expandafter\let\csname old#1\expandafter\endcsname\csname #1\endcsname
  \expandafter\let\csname oldend#1\expandafter\endcsname\csname end#1\endcsname
  \renewenvironment{#1}{\linenomath\csname old#1\endcsname}{\csname oldend#1\endcsname\endlinenomath}}\newcommand*\patchBothAmsMathEnvironmentsForLineno[1]{\patchAmsMathEnvironmentForLineno{#1}\patchAmsMathEnvironmentForLineno{#1*}}\AtBeginDocument{\patchBothAmsMathEnvironmentsForLineno{equation}\patchBothAmsMathEnvironmentsForLineno{align}\patchBothAmsMathEnvironmentsForLineno{flalign}\patchBothAmsMathEnvironmentsForLineno{alignat}\patchBothAmsMathEnvironmentsForLineno{gather}\patchBothAmsMathEnvironmentsForLineno{multline}}

\begin{document}

\author[J.\ Konieczny]{Jakub Konieczny}
\address[J.\ Konieczny]{Camille Jordan Institute, 
Claude Bernard University Lyon 1,
43 Boulevard du 11 novembre 1918,
69622 Villeurbanne Cedex, France}
\email{jakub.konieczny@gmail.com}

\title{Characterisation of Meyer sets via the Freiman--Ruzsa theorem}

\begin{abstract}	
	We show that the  Freiman--Ruzsa theorem, characterising finite sets with bounded doubling, leads to an alternative proof of a characterisation of Meyer sets, that is, relatively dense subsets of Euclidean spaces whose difference sets are uniformly discrete.  
\end{abstract}

\subjclass[2010]{}

\maketitle

\newcommand{\ee}{d}
\newcommand{\dd}{e}	

\section{Introduction}\label{sec:Intro}

The discovery of \emph{quasicrystals} \cite{Shechtman-1984} (see also \cite{LevineSteinhardt-1984,KramerNeri-1984}) sparked considerable  interest in non-periodic, discrete subsets of $\RR^{\ee}$ which exhibit approximate symmetry. For extensive background, we refer to the monograph \cite{BaakeGrimm-book-1,BaakeGrimm-book-2}; accessible introduction can also be found e.g.{} in \cite{Lagarias-1999}, \cite{Lagarias-2000}, \cite{Baake-2002}. More recently, non-abelian analogues have been investigated in \cite{BjorklundHartnickPogorzelski-2018, BjorklundHartnickPogorzelski-2021, BjorklundHartnickPogorzelski-2022}.

 Meyer sets, originally introduced by Meyer  \cite{Meyer-book-1970, Meyer-book-1972} and given their name by Moody \cite{Moody-1995}, are one of the fundamental concepts in the mathematical theory of quasicrystals. They are defined as sets satisfying any of the equivalent conditions in the following theorem. 

\begin{theorem}\label{thm:char_of_Meyer}
	Let $\ee \geq 1$ and let $A \subset \RR^{\ee}$ be a relatively dense set. Then the following conditions are equivalent:
	\begin{enumerate}
	\item\label{it:01:A} $D^+(A-A) < \infty$;
	\item\label{it:01:B} $A - A$ is uniformly discrete;
	\item\label{it:01:C} $A$ is uniformly discrete and $A - A \subset A + F$ for a finite set $F$;
	\item\label{it:01:D} $A \subset M+F$ for a cut-and-project set $M$ and a finite set $F$.
	\end{enumerate}
\end{theorem}
Let us briefly review the terminology used above; more extensive and precise discussion can be found in Sections \ref{sec:Additive} and \ref{sec:Quasicrystals}. The upper uniform asymptotic density of a set $A \subset \RR^d$ is given by 
\[ D^+(A) = \limsup_{R \to \infty} \sup_{x \in \RR^d} (2R)^{-d} \abs{ \set{ a \in A}{ \norm{a - x}_\infty < R}}. \] The difference set of  a set $A$ is given by $A-A = \set{a-b}{a,b \in A}$. A set $A$ is relatively dense if it intersects each ball of radius $R$ for some $R > 0$, and it is uniformly discrete if for some $r > 0$ each ball with radius $r$ contains at most one point in $A$. A cut-and-project set takes the form $\pi_1( \Gamma \cap (\RR^d \times \Omega))$ for some lattice $\Gamma  \subset \RR^{d+e}$, some bounded open set $\Omega \subset \RR^e$ and the standard projection $\pi_1 \colon \RR^{d+e} \to \RR^{d}$. 

It is easy to verify the chain of implications \eqref{it:01:D} $\Rightarrow$ \eqref{it:01:C} $\Rightarrow$ \eqref{it:01:B} $\Rightarrow$ \eqref{it:01:A}, so the substance of the above theorem is that \eqref{it:01:A} $\Rightarrow$ \eqref{it:01:D}. Implication \eqref{it:01:B} $\Rightarrow$ \eqref{it:01:C} was shown by Lagarias \cite[Thm.\ 1.1]{Lagarias-1996}, and later a short proof of \eqref{it:01:A} $\Rightarrow$ \eqref{it:01:C} was found by Lev and Olevskii \cite[Lem.\ 8]{LevOlevskii-2015}. The most difficult transition appears to be \eqref{it:01:C} $\Rightarrow$ \eqref{it:01:D} which due to Meyer \cite[Chpt.\ II, Thm. IV]{Meyer-book-1972}, see also \cite{Moody-1997}.

The above theorem bears a striking resemblance to the Freiman--Ruzsa theorem \cite{Freiman-book,Ruzsa-1989}, which is one of the cornerstones of additive combinatorics. In additive combinatorics one is primarily interested in the behaviour of finite subsets of abelian groups under addition. The Freiman--Ruzsa theorem provides a characterisation of finite sets $A$ such that the sumset $A+A = \set{a+b}{a,b \in A}$ is not much larger than $A$. 
\begin{theorem}[Freiman--Ruzsa]\label{thm:Freiman-Ruzsa}
	Let $A$ be a finite, non-empty subset of an abelian torsion-free group. Consider the conditions:
\begin{enumerate}
	\item[(i)$_{K}$] $\abs{A+A} < K \abs{A}$ for a constant $K \geq 2$;
	\item[(ii)$_{d,C}$] $A$ is contained in a generalised arithmetic progression $P$ of rank $d$ and size $\leq C \abs{A}$, for some constants $d \in \NN$ and $C$.
\end{enumerate}	
Then (i) and (ii) are equivalent in the following sense: For each $K$ there exist $d = d(K)$ and $C = C(K)$ such that (i)$_{K}$ implies (ii)$_{d,C}$. Conversely, for each $C,d$ there exists $K = K(C,d)$ such that (ii)$_{d,C}$ implies (i)$_{K}$. 
\end{theorem}
A generalised arithmetic progression of rank $d$ is a set of the form $P_1 + P_2 + \dots + P_d$ where $P_i$ are arithmetic progressions, for details see Section \ref{sec:Additive}.
The implication (ii) $\Rightarrow$ (i) is, again, comparatively simple, so the content of the theorem is the reverse implication (i) $\Rightarrow$ (ii). 

Interestingly, despite the similarity between Theorems \ref{thm:char_of_Meyer} and \ref{thm:Freiman-Ruzsa}, their proofs use rather different techniques. Hence, it is natural to inquire if the similarity is only superficial or if there is a more tangible connection. This line of inquiry was suggested by Lev and Olevskii \cite[Sec.{} 11.3]{LevOlevskii-2017}. The purpose of this paper is to exhibit such a connection by showing that Theorem \ref{thm:Freiman-Ruzsa} together with some basic tools in additive combinatorics implies Theorem \ref{thm:char_of_Meyer} in a relatively elementary way. 

Additionally, we point out that quantitative variants of Theorem  \ref{thm:Freiman-Ruzsa} have been extensively investigated by many authors, culminating in the work of Sanders \cite{Sanders-2013}. Conversely, quantitative bounds in Theorem \ref{thm:char_of_Meyer} are rarely addressed. Our argument preserves quantitative bounds, and as a consequence we obtain the following new result. We use $D^{-}$ to denote lower uniform asymptotic density:
 \[ D^-(A) = \liminf_{R \to \infty} \inf_{x \in \RR^d} (2R)^{-d} \abs{ \set{ a \in A}{ \norm{a - x}_\infty < R}}. \]

\begin{theorem}\label{thm:main-quant}
	Suppose that $A \subset \RR^{\ee}$ is relatively dense and $D^{+}(A-A) < KD^{-}(A)$ for some $K \geq 2$. Then $A \subset M+F$, where $F$ is a finite set and $M$ is a cut-and-project set with at most $O\bra{\bra{d + \log K}^7}$ internal dimensions.
\end{theorem}

The exponent $7$ is not optimal, and can be replaced with any value strictly larger than $6$. Further improvements would follow from stronger quantitative versions of the Freiman--Ruzsa theorem. In particular, it is plausible that $O(d+\log K)$ internal dimensions suffice. See Section \ref{sec:Epilogue} for further discussion.

\subsection*{Acknowledgements}
The author is grateful to Nir Lev for introducing him to this subject and for spending a considerable amount of time on fruitful conversations, and to Tom Sanders for informative comments on the quantitative variants of the Freiman--Ruzsa theorem. The author also thanks  Michal Kupsa and Rudi Mrazovi\'{c} for helpful remarks, and the anonymous referee for careful reading of this paper and valuable corrections.

While writing this paper, the author was supported by the ERC grant ErgComNum 682150 at the Hebrew University of Jerusalem. Currently, the author is working within the framework of the LABEX MILYON (ANR-10-LABX-0070) of Universit\'{e} de Lyon, within the program "Investissements d'Avenir" (ANR-11-IDEX-0007) operated by the French National Research Agency (ANR). The author also acknowledges support from the Foundation for Polish Science (FNP).

\section{Additive combinatorics}\label{sec:Additive}
\subsection{Basic definitions} Let $Z$ be an abelian group. For sets $A, B \subset Z$, we let $A+B$ denote the \emph{sumset} $\set{a+b}{a \in A,\ b \in B}$, and similarly we let $A-B = \set{a-b}{a \in A,\ b \in B}$ denote the \emph{difference set}. For and integer $k \geq 1$, we let $kA$ denote the $k$-fold sumset of $A$, whence $1 A = A$ and $(k+1)A = A + kA$. The \emph{doubling} of $A$ (assuming additionally that $A \neq \emptyset$) is the ratio $\abs{2A}/\abs{A}$. If $Z$ is torsion-free then it is well known that $\abs{2A} \geq 2\abs{A}-1$ (see e.g.\ \cite[Thm.\ 5.5]{TaoVu-book}). 

A \emph{generalised arithmetic progression of rank $d$} in $Z$ (or GAP for short) is a subset $Z$ of the form
\begin{equation}\label{eq:38:00}
	P = \set{ n_1a_1 + n_2 a_2 + \dots + n_d a_d + b }{ 0 \leq n_i \leq l_i \text{ for all } 1 \leq i \leq d}
\end{equation}
for some $a_1,a_2,\dots,a_d,b \in Z$ and $l_1,l_2,\dots,l_d \in \NN_0$. Whenever we speak of a GAP, we always have in mind a fixed choice of the steps $a_i$, base point $b$ and side lengths $l_i$. The \emph{size} of the GAP $P$ given by \eqref{eq:38:00} is defined as
\begin{equation}
	\norm{P} = \prod_{i=1}^d (l_i+1),
\end{equation}
and may in general be larger than the cardinality $\abs{P}$ (i.e., the number of distinct elements of $P$). The GAP $P$ is \emph{proper} if $\norm{P} = \abs{P}$, or equivalently if all of the elements $n_1a_1 + n_2 a_2 + \dots + n_d a_d + b$ in \eqref{eq:38:00} are distinct. We note that the set $P$ given by \eqref{eq:38:00} is also a GAP of rank $d+1$ (with a side of length $0$). The GAP $P$ has doubling bounded in terms of the rank: $\abs{2P} \leq 2^d \abs{P}$.
We will also be interested in \emph{symmetric} GAPs, which take the form
\begin{equation}\label{eq:38:01}
	Q = \set{ n_1a_1 + n_2 a_2 + \dots + n_d a_d}{ \abs{n_i} \leq l_i \text{ for all } 1 \leq i \leq d},
\end{equation}
for some $a_1,a_2,\dots,a_d \in Z$ and $l_1,l_2,\dots,l_d \in \NN_0$. Of course, \eqref{eq:38:01} defines a GAP (with side lengths $2l_i+1$ and base point $-\sum_{i=1}^{d} l_i a_i $). 

\subsection{Freiman--Ruzsa}

We record a classical inequality due to Pl\"{u}nnecke \cite{Plunnecke-1970}, later refined by Ruzsa \cite{Ruzsa-1989} and used as a key step in the proof of the Freiman--Ruzsa theorem; a considerably shorter proof was subsequently found by Petridis \cite{Petridis-2012}.
\begin{theorem}\label{thm:Plunnecke} Fix $K \geq 1$.
	Let $A,B$ be finite subsets of an abelian group $Z$, and suppose that $\abs{A+B} \leq K \abs{A}$. Then $\abs{kB-lB} \leq K^{k+l}\abs{A}$ for each $k,l \geq 0$. 
\end{theorem}

We will need the following variant of the Freiman--Ruzsa theorem, which we will apply with $B = -A$.

\begin{theorem}\label{thm:FR}
	Fix $K \geq 2$. Let $A,B$ be finite subsets of a torsion-free abelian group $Z$, and suppose that $\abs{A} = \abs{B} = n$ and $\abs{A+B} \leq Kn$ for some $n > 0$. Then there exists a proper symmetric GAP $Q$ of rank $O_K(1)$ and cardinality $\abs{Q} = O_K(n)$ as well as  a finite set $F$ of cardinality $\abs{F} = O_K(1)$ such that $Q \subset 2A-2A$ and $A \subset F+Q$.
\end{theorem}

This formulation is slightly different than the results that can be found in \cite{Freiman-book,Ruzsa-1989}. In the symmetric case (that is, $A = B$), this follows directly from \cite[Thm.{} 5.32]{TaoVu-book}. 
To deduce the asymmetric case from the symmetric one, it suffices to notice that, as a special case of the Pl\"{u}nnecke inequality (Theorem \ref{thm:Plunnecke}), if  $\abs{A} = \abs{B} = n$ and $\abs{A+B} \leq Kn$  then $\abs{2A} \leq K^2 n$.

\subsection{Continuous variants}
Lastly, we briefly discuss set addition in the continuous setting. We let $\lambda_X$ denote the Lebesgue measure on the space $X$, where $X$ is either $\RR^d$ or $\RR^d/\ZZ^d$ for some $d \in \NN$. When there is no risk of confusion, we write $\lambda$ in place of $\lambda_X$. The following theorem is due to Macbeath \cite{Macbeath-1953} and was later generalised to arbitrary compact groups by Kemperman \cite{Kemperman-1964}, see also \cite[Ex.{} 5.1.12]{TaoVu-book}. For further discussion, we also refer to \cite{Bilu-2012}.

\begin{theorem}\label{thm:Kemperman}
	Let $A,B \subset \RR^d/\ZZ^d$ be open. Then $\lambda(A+B) \geq \min\bra{\lambda(A)+\lambda(B),1}$.
\end{theorem}

We point out that for open $A,B \subset \RR^d/\ZZ^d$, the set $A+B$ is again open and hence in particular measurable. Similarly, if $A,B \subset \RR^d/\ZZ^d$ are closed then $A+B$ is closed. On the other hand, it is a classical result that $A+B$ need not be measurable when $A$ and $B$ are measurable \cite{Sierpinski-1920}.  

By a parallelepiped in $\RR^d$ we mean a set of the form
\begin{align*}
	P &= x + [0,1]v_1 + [0,1]v_2 + \dots + [0,1]v_d
	\\& = \set{ x + t_1 v_1 + t_2 v_2 + \dots + t_d v_d}{t_1,t_2,\dots,t_d \in [0,1]},
\end{align*}
where $x,v_1,v_2,\dots,v_d \in \RR^d$ and $v_1,v_2,\dots,v_d$ are linearly independent. (Thus, for the purposes of this paper, parallelepipeds are closed and non-degenerate.) We will use the following elementary consequence of the theorem above. 

\begin{corollary}\label{cor:Kemperman}
	Let $d \in \NN$, $\e \in (0,1)$, let $P \subset \RR^d$ be a parallelepiped, and let $A \subset P$ be a closed set with $\lambda(A) \geq \e \lambda(P)$ and $\lambda( A \setminus \inter A) = 0$. Then for each integer $k \geq d\floor{8/\e}$ there exists a vector $b \in \RR^d$ such that $kA \supset P+b$.
\end{corollary}
\begin{proof}
	Applying an affine transformation, we may assume that $P = [0,1]^d$.

	Suppose first that $d = 1$. Replacing $A$ with a translate, we may assume that $0 \in A$. Put $t := \max A$, $l := \ceil{t/\e}$, and let $\bar A$ be the image of $A$ in $\RR/t\ZZ$. By Theorem \ref{thm:Kemperman}, the set $l (\inter \bar A)$ has full measure. Since $l \bar A$ is closed it follows that $l \bar A = \RR/t\ZZ$. Hence, for each $x \in [0,t]$ there exists $y \in l A$ and $0 \leq m \leq l$ such that 
	\[
		x = y - mt = y + (l - m)t - l t \in (2l - m) A - l t \subset 2l A - l t.
	\]
	Thus, we have
	\(
		2l A \supset [0,t] + l t.
	\)
	Let $k \geq 2\ceil{1/t} l$ and put $b := \ceil{1/t}lt$. Then 
	\[
		k A \supset \ceil{1/t} \bra{2l A} \supset \ceil{1/t} \bra{[0,t]+lt} \supset [0,1] + b.
	\] 
	It remains to note that $2\ceil{1/t} l$ is an integer with $2\ceil{1/t} l \leq 2 \cdot (2/t) \cdot (2t/\e) \leq 8/\e$.
	
	Next, consider the case where $d \geq 2$. For each $1 \leq i \leq d$, we can find a line $\ell_i$ parallel to $\vec e_i$, the $i$-th vector in the standard basis of $\RR^d$, such that $\lambda_{\ell_i}(A \cap \ell_i) \geq \e$. Put $k_1 := \floor{8/\e}$. It follows from the $1$-dimensional case that $k_1 A$ contains a unit interval $ [0,1]\vec e_i + b_i$ for some $b_i \in \RR^d$. Let $z \in A$ be an arbitrary point, let $k \geq d k_1$ and put $b := \sum_{i=1}^d b_i + (k-dk_1)z$. Then
	\[
		kA  \supset [0,1]\vec e_1 + [0,1]\vec e_2 + \dots + [0,1]\vec e_d + b = [0,1]^d + b. \qedhere
	\] 
\end{proof}

\section{Quasicrystals}\label{sec:Quasicrystals}
\subsection{Basic definitions}
Let $X$ be a metric space, and let $A$ be a subset of $X$. We say that $A$ is $R$-\emph{relatively dense} ($R > 0$) if each closed ball of radius $R$ contains a point in $A$, that is, $\bBall{X}{x}{R} \cap A \neq \emptyset$ for all $x \in X$. Accordingly, $A$ is relatively dense if it is $R$-relatively dense for some $R > 0$, and the infimum of the admissible values of $R$ is called the \emph{covering radius}. In a similar vein, $A$ is $r$-\emph{uniformly discrete} ($r>0$) if each open ball with radius $r$ contains at most one point of $A$: $\abs{\Ball{X}{x}{r} \cap A} \leq 1$ for all $x \in X$. Also, $A$ is uniformly discrete if it is $r$-uniformly discrete for some $r>0$. If $A$ is both relatively dense and uniformly discrete then it is called a \emph{Delone set}.

For the sake of concreteness, we will from now on restrict to the case where $X = \RR^d$ for some $d \in \NN$, equipped with the supremum norm. (The choice of the norm does not play a significant role, but working with the supremum norm is slightly more convenient.) We define the upper and lower uniform density of a set $A \subset \RR^d$ as
\begin{align*}
	D^+(A) = \limsup_{L \to \infty} \sup_{x \in \RR^d} \frac{ \abs{A \cap \BallR{d}{x}{L}} }{ \lambda \bra{\BallR{d}{x}{L}}},\quad 
	D^-(A) = \liminf_{L \to \infty} \inf_{x \in \RR^d} \frac{ \abs{A \cap \BallR{d}{x}{L}} }{ \lambda \bra{\BallR{d}{x}{L}}},
\end{align*}
where $\BallR{d}{x}{L} = \prod_{i=1}^d (x_i-L,x_i+L)$. If $A$ is $R$-relatively dense then $D^-(A) \geq 1/(2R)^d > 0$, which can be easily verified by noticing that $\BallR{d}{x}{L}$ contains a collection of $(L/R)^d - O_{d,R}((L/R)^{d-1})$ pairwise disjoint open balls $\BallR{d}{y}{R}$ with $y \in \BallR{d}{x}{L}$. Likewise, if $A$ is 
 $r$-uniformly discrete then $D^+(A) \leq 1/r^d < \infty$.

\subsection{Limits of sets}\label{ssec:SetLimit}
We will need a suitable notion of convergence of sets. For a sequence of sets $A_n \subset \RR^d$ ($n \in \NN$) we will say that a set $A \subset \RR^d$ is the \emph{Kuratowski limit} of $A_n$ as $n \to \infty$, denoted $A = \klim_{n \to \infty} A_n$, if the following two conditions are satisfied. (Similar notion of convergence appears e.g.{} in \cite{Solomyak-1995}.)
\begin{enumerate}[wide]
\item For each $x \in A$, for each open neighbourhood $x \in U \subset \RR^d$, for all sufficiently large $n \in \NN$ we have $A_n \cap U \neq \emptyset$.
\item For each $x \in \RR^d \setminus A$, there exists an open neighbourhood $x \in U \subset \RR^d$ such that for all sufficiently large $n \in \NN$ we have $A_n \cap U = \emptyset$.
\end{enumerate}
We will review only the basic properties of Kuratowski limits and refer e.g.{} to \cite{Beer-1993} for further details.  
The Kuratowski limit is not guaranteed to exist, but when it does it is unique. Additionally, the Kuratowski limit, if it exists, is a closed set. Moreover, each sequence of subsets of $\RR^d$ has a convergent subsequence (cf.{} \cite[Thm.{} 5.2.11]{Beer-1993}) 

\begin{lemma}\label{lem:klim-props}
	Let $A_n \subset \RR^d$ ($n \in \NN$) be a sequence of sets and let $A = \klim_{n \to \infty} A_n$.
\begin{enumerate}
\item Let $r > 0$. If $A_n$ is $r$-uniformly discrete for each $n \in \NN$ then also $A$ is  $r$-uniformly discrete.
\item Let $R > 0$. If $A_n$ is $R$-relatively dense for each $n \in \NN$ then also $A$ is  $R$-relatively dense.
\item If $A_n$ is a group for each $n \in \NN$ then also $A$ is a group.
\end{enumerate} 
\end{lemma}
\begin{proof}
\begin{enumerate}
\item For the sake of contradiction, suppose an open ball $\BallR{d}{x}{r}$ contained two points $y,y' \in A \cap \BallR{d}{x}{r}$. Let $U,U' \subset \BallR{d}{x}{r}$ be two open, disjoint neighbourhoods of $y,y'$ respectively. For sufficiently large $n$ we have $A_n \cap U \neq \emptyset$ and $A_n \cap U' \neq \emptyset$. Hence, $\abs{A_n \cap \BallR{d}{x}{r}} \geq 2$, contrary to the assumption that $A_n$ are $r$-uniformly discrete.
\item Consider any $x \in \RR^d$. Since $A_n$ is $R$-relatively dense for each $n$, there exist points $y_n \in A_n \cap \bBallR{d}{x}{R}$. Let $y$ be an accumulation point of $y_n$. Directly by the definition of the Kuratowski limit, we see that $y \not \in \RR^d \setminus A$, and thus $y \in A$. It remains to notice that $y \in \bBallR{d}{x}{R}$.
\item It is enough to show that for each $x,y \in A$ also $x-y \in A$. By definition of the Kuratowski limit, there exist sequences $x_n,y_n \in A_n$ ($n \in \NN$) such that $x_n \to x$ and $y_n \to y$ as $n \to \infty$. Since $A_n$ is a group, $x_n - y_n \in A_n$. It remains to note that $x_n - y_n \to x-y$ as $n \to \infty$, and hence $x-y \in A$.  
 \qedhere
\end{enumerate}
\end{proof}

\subsection{Cut-and-project sets}
A \emph{lattice} in $\RR^d$ is a subgroup of $\RR^d$ that is both relatively dense and (uniformly) discrete, or equivalently a subgroup of $\RR^d$ generated by $d$ linearly independent vectors. For a more detailed discussion, see e.g.{} \cite[Sec.{} 3]{TaoVu-book}.
For the purposes of this paper, we define \emph{cut-and-project} sets in $\RR^d$ to be the sets of the form
\begin{equation}
\label{eq:78:00}
	M = \pi_1\bra{\Gamma \cap \bra{ \RR^d \times \Omega}}
\end{equation}
where $\pi_1 \colon \RR^d \times \RR^e \to \RR^d$ is the projection onto the first coordinate, $e \in \NN_0$, $\Gamma$ is a discrete subgroup of $\RR^{d+e}$ and $\Omega \subset \RR^e$ is open and bounded. 
When speaking of a cut-and-project set, we always have in mind a representation as in \eqref{eq:78:00}. For later reference, $\pi_2 \colon \RR^d \times \RR^e \to \RR^e$ is defined accordingly as the projection onto the second coordinate. We will occasionally call $e$ the number of \emph{internal} dimensions of $M$. We note that sets of the form \eqref{eq:78:00} are also considered under different topological assumptions on $\Omega$; for instance \cite{Lagarias-2000} requires $\Omega$ to be compact. Since we are ultimately interested in showing that a given set is contained in a bounded number of translates of a cut-and-project set, we may freely replace $\Omega$ with any larger set, and thus the topological properties of $\Omega$ do not play a major role in our main result. We emphasise that we do not impose any of the following commonly included conditions:
\begin{enumerate}[label=(\Roman*),ref=\Roman*]
\item\label{it:00:A} $\Gamma$ is a lattice;
\item\label{it:00:B} $\pi_1|_{\Gamma}$ is injective;
\item\label{it:00:C} $\pi_2(\Gamma)$ is dense in $\RR^e$.
\end{enumerate}
In particular, a cut-and-project set may be finite or even empty. 

We also define a \emph{shift-cut-and-project} set to be any set of the form $F+M$ where $F$ is a finite set and $M$ is a cut-and-project set. Allowing for a finite number of shifts lets us ensure conditions  \eqref{it:00:A}, \eqref{it:00:B}, \eqref{it:00:C} under additional mild assumptions, as the following proposition shows. In particular, in Theorem \ref{thm:char_of_Meyer}\eqref{it:01:D} one may freely assume that the set $M$ satisfies \eqref{it:00:A}, \eqref{it:00:B}, \eqref{it:00:C}.

\begin{proposition}\label{lem:cut-and-project-WLOG}
	Let $A \subset \RR^d$ be a relatively dense shift-cut-and-project set. Then $A$ is contained in a shift-cut-and-project set $F+M$ where $F$ is finite and $M$ is a cut-and-project set satisfying \eqref{it:00:A}, \eqref{it:00:B} and \eqref{it:00:C}.
\end{proposition}
\begin{proof}
	Among all shift-cut-and-project sets $F+M$ containing $A$ with $F$ finite and $M$ given by \eqref{eq:78:00}, pick one with the least possible number of internal dimensions $e$. We will show that $M$ satisfies \eqref{it:00:A}, \eqref{it:00:B} and \eqref{it:00:C}.
	
	\textit{Item \eqref{it:00:A}:}
	Suppose first that $\Gamma$ was not a lattice. Since $\Gamma$ is assumed to be a discrete subgroup, it follows that $\Gamma$ is not relatively dense. Hence, $\Gamma$ is contained in some codimension $1$ subspace $U_0 < \RR^d \times \RR^e$. Let $v \in \RR^d \times \RR^e$ be a unit vector orthogonal to $U_0$, meaning that $U_0 = \set{u \in \RR^d \times \RR^e}{ \left< u,v \right> = 0}$. 
	
	For $x \in \RR^d$, let $\gamma(x)$ be the element of $\Gamma$ that is closest to $(x,0) \in \RR^d \times \RR^e$ (if there are two element at the same distance, we break the tie arbitrarily). Our next goal is to show that $(x,0) - \gamma(x)$ is bounded uniformly for all $x \in \RR^d$. Let $R,R',R'' > 0$ be such that $A$ is $R$-relatively dense in $\RR^d$, $F \subset \BallR{d}{0}{R'}$ and $\Omega \subset \BallR{e}{0}{R''}$. For each $x \in \RR^d$ we can find $x' \in A $ with $\norm{x-x'}_{\infty} \leq R$. Since $A \subset F+M$, we can find $x'' \in M$ with $\norm{x'-x''}_{\infty} < R'$. By definition \eqref{eq:78:00}, there is $\gamma \in \Gamma \cap (\RR^d \times \Omega)$ such that $x'' = \pi_1(\gamma)$. Hence, $\norm{(x'',0) - \gamma}_{\infty} = \norm{\pi_2(\gamma)}_{\infty} < R''$. Combining the inequalities above, we conclude that 
	$\norm{(x,0)-\gamma}_{\infty} < R+R'+R''$, as needed.
	
	Next, we observe that $v \in \{0\} \times \RR^e$. Indeed, for any $x \in \RR^d$ we can compute that
	\[
	 \left< (x,0),v \right> = \lim_{t \to \infty} \frac{1}{t} \left< (tx,0),v \right>  = \lim_{t \to \infty} \frac{1}{t} \left< (tx,0) - \gamma(tx),v \right> = 0,
	\]	 
	and thus $v \in \bra{ \RR^d \times \{0\}}^{\perp} = \{0\} \times \RR^e$. It follows that $U_0$ takes the form $U_0 = \RR^d \times U$ where $U = \set{y \in \RR^e}{ \left< y, \pi_2(v) \right> = 0}$ is the orthogonal complement of $\pi_2(v)$. Let $\Omega' = \Omega \cap U$ and note that $\Omega'$ open and bounded.
Since $\Gamma \subset \RR^d \times U$, we have
\begin{equation}
\label{eq:78:01}
	M = \pi_1\bra{\Gamma \cap \bra{\RR^d \times \Omega}} = \pi_1'\bra{\Gamma \cap \bra{ \RR^d \times \Omega'} },
\end{equation}
where $\pi_1' \colon \RR^d \times U \to \RR^d$ is the projection onto the first coordinate. Since $e' := \dim U = e-1$, this contradicts the minimality of $e$.

	\textit{Item \eqref{it:00:B}:}
Suppose next that $\pi_1$ was not injective on $\Gamma$. Since $\Gamma$ is a group, this means that $0 \in \pi_1(\Gamma \setminus \{0\})$. Thus, $\Gamma$ contains an element of the form $\gamma_1 = (0,\delta) \in \RR^d \times \RR^e$.
Let $U = \set{y \in \RR^e}{ \left< y, \gamma_1 \right> = 0}$ denote the orthogonal complement of $\gamma_1$ in $\RR^e$, and let $\rho \colon \RR^e \to U$ be the orthogonal projection. Let $\tilde\rho \colon \RR^d \times \RR^e \to \RR^d \times U$ denote the map $\operatorname{id} \times \rho$, meaning that $\tilde\rho(x,y) = (x,\rho(y))$ for $x \in \RR^d$ and $y \in \RR^e$. Let $\Omega' = \rho(\Omega)$ and $\Gamma' = \tilde\rho(\Gamma)$. 

We note that $\Omega'$ is open and bounded, since it is the image of an open and bounded set under a surjective linear map. Likewise, $\Gamma'$ is relatively dense and it is a group, since these properties are also preserved under surjective linear maps. Replacing $\gamma_1$ with a scalar multiple $\frac{1}{n}\gamma_1$ ($n \in \NN$) if necessary, we may assume that $\gamma_1$ can be completed to basis $\gamma_1,\gamma_2,\dots,\gamma_{d+e}$ of $\Gamma$ (i.e., $\Gamma = \ZZ\gamma_1 + \ZZ \gamma_2 + \dots + \ZZ \gamma_{d+e}$). Bearing in mind that $\tilde\rho(\gamma_1) = 0$ we see that $\Gamma'$ is generated by the $d+e-1$ vectors $\tilde\rho(\gamma_2),\tilde\rho(\gamma_3),\dots \tilde\rho(\gamma_{d+e})$, which are easily checked to be linearly independent (otherwise $\Gamma'$ would not be relatively dense). It follows that $\Gamma'$ is discrete, and thus it is a lattice.

For each $x \in M$, directly by definition \eqref{eq:78:01}, we can find $\gamma \in \Gamma \cap (\RR^d \times \Omega)$ with $x = \pi_1(\gamma)$. Let $\gamma' = \tilde\rho(\gamma)$. Note that $\gamma' \in \Gamma'$ since $\gamma \in \Gamma$ and that $\gamma' \in \tilde\rho(\RR^d \times \Omega) = \RR^d \times \Omega'$ since $\gamma \in \RR^d \times \Omega$. Let $\pi_1' \colon \RR^d \times U \to \RR^d$ denote the projection onto the first coordinate. Since $\pi_1' = \pi_1 \circ \tilde \rho$, we have $x = \pi_1' (\gamma')$. Since $x$ was arbitrary, 
\begin{equation}
\label{eq:78:02}
	M \subset \pi_1'\bra{\Gamma' \cap  \bra{\RR^d \times \Omega'}}.
\end{equation}
Hence, we again obtain a contradiction with minimality of $e$.

\textit{Item \eqref{it:00:C}:}
Finally, suppose that $\pi_2(\Gamma)$ was not dense in $\RR^e$. Since $\pi_2(\Gamma)$ is a non-dense subgroup of $\RR^e$, there exists a  non-zero vector $v \in \{0\} \times \RR^e$ such that $\left<v,\Gamma \right> \subset \ZZ$ (see Lemma \ref{lem:lin-alg} below for details). Since $\Gamma$ is relatively dense, so is $\left<v,\Gamma \right>$, and in particular $\left<v,\Gamma \right> \neq \{0\}$. It follows that $\left<v,\Gamma \right> = m \ZZ$ for some $m \in \NN$, and replacing $v$ with $\frac{1}{m}v$ we may freely assume that $m = 1$. Let $\gamma_1 \in \Gamma$ be such that $\left< v, \gamma_1 \right> = 1$, and let $U \subset \RR^e$ be the orthogonal complement of $\pi_2(v)$ so that $\RR^d \times U$ is the orthogonal complement of $v$. Let $\Gamma' = \Gamma \cap  \bra{\RR^d \times U}$. Note that each $\gamma \in \Gamma$ can be decomposed as $\gamma = \gamma' + \left<v,\gamma\right> \gamma_1$, where $\gamma' = \gamma - \left<v,\gamma\right>\gamma_1 \in \Gamma'$.

Pick any $x \in M$. Then there exists a point $y \in \Omega$ such that $(x,y) \in \Gamma$. We may write $(x,y) = \gamma' + n \gamma_1$ where $\gamma' \in \Gamma'$ and $n \in \ZZ$ is given by
\(
	n = \left<v,(x,y) \right> = \left<\pi_2(v),y \right>.
\)
In particular, 
\( \abs{n} \leq n_{\max} := \max_{z \in \Omega} \abs{\left<\pi_2(v), z \right>}
\)
is bounded uniformly with respect to $x$. 
It follows that $x = \pi_1(\gamma') + n \pi_1(\gamma_1)$, where $\gamma' \in \Gamma'$ and $\pi_2(\gamma') = y - n \pi_2(\gamma_1) \in U \cap(\Omega - n\pi_2(\gamma_1))$. Let
\[ E = \set{ h \pi_1(\gamma_1)}{  \abs{h} \leq n_{\max} } \quad \text{and} \quad \Omega' = \bigcup_{\abs{h} \leq n_{\max} } U \cap(\Omega - h \pi_2(\gamma_1)).\]
Since $x \in M$ was arbitrary, we conclude that 
\begin{equation}
\label{eq:78:03}
	M \subset E + \pi_1'\bra{\Gamma' \cap  \bra{\RR^d \times \Omega'}},
\end{equation}
once again contradicting minimality of $e$.
\end{proof}

\begin{lemma}\label{lem:lin-alg}
	Let $d,e \in \NN$, and let $u_1,u_2,\dots,u_{d+e} \in \RR^e$ be vectors such that the group $G:= \ZZ u_1 + \ZZ u_2 + \cdots + \ZZ u_{d+e}$ which they generate is not dense in $\RR^e$. Then there exists a non-zero vector $v \in \RR^e$ such that $\left<u_i, v\right> \in \ZZ$ for all $1 \leq i \leq d+e$.
\end{lemma}
\begin{proof}
	We may assume that the vectors $u_1,u_2,\dots,u_{d+e}$ span $\RR^e$, since otherwise there would exist a non-zero vector $v \in \RR^e$ orthogonal to all of $u_1,u_2,\dots,u_{d+e}$. Rearranging the vectors if necessary, we may assume that $u_{d+1}, u_{d+2}, \dots, u_{d+e}$ are a basis of $\RR^{e}$. Applying a change of coordinates, we may further assume that for each $1 \leq i \leq e$, $u_i$ is the $i$-th standard basis vector of $\RR^e$, which in particular implies that $\ZZ u_{d+1} + \ZZ u_{d+2} + \dots + \ZZ u_{d+e} = \ZZ^{e}$. For $1 \leq i \leq d$, let $\bar u_i$ be the image of $u_i$ in the torus $\RR^e/\ZZ^e$. Since $G$ is not dense in $\RR^e$, the group $\bar G = \ZZ \bar u_1 + \ZZ \bar u_2 + \cdots + \ZZ \bar u_{d}$ is not dense in $\RR^e/\ZZ^e$. It follows from the multiparameter variant of Kronecker's equidistribution theorem (e.g.\ \cite[Thm.\ 6.4]{KuipersNiederreiter-book}) that there exists a non-zero vector $v \in \ZZ^d$ with $\left< \bar u_i, v \right> = 0$, meaning that $\left< u_i, v \right> \in \ZZ$, for all $1 \leq i \leq d$. Since $v \in \ZZ^d$, also $\left< u_{d+i}, v \right> \in \ZZ$ for all $1 \leq i \leq e$.
\end{proof}
 
\section{Proof of the main theorem}
	
	We are now ready to prove our main result, which is tantamount to the implication  \eqref{it:01:A} $\Rightarrow$ \eqref{it:01:D} in Theorem \ref{thm:char_of_Meyer}.
	
\begin{theorem}\label{thm:main}
	Let $A \subset \RR^\ee$ for some $d \in \NN$ and suppose that $A$ is relatively dense and that $D^{+}(A-A) < \infty$. Then $A$ is contained in a shift-cut-and-project set.
\end{theorem} 
\begin{proof}
 Rescaling $A$ if necessary, we may assume that $A$ is $1$-relatively dense. Similarly, replacing $A$ with a translate, we may assume that $0 \in A$. For $N \in \NN$, let $A_N := A \cap \BallR{\ee}{0}{N}$. We next observe that the sets $A_N$ have bounded doubling:
 \[ 
 \limsup_{N \to \infty} \frac{\abs{A_N - A_N}}{\abs{A_N}} \leq 
\limsup_{N \to \infty} \frac{\abs{(A-A) \cap \BallR{\ee}{0}{2N}}}{\abs{A_N}} \leq \frac{2^d D^{+}(A-A)}{D^{-}(A)}.\]
Thus, there exists a constant $K > 0$ such that $\abs{A_N-A_N} \leq K \abs{A_N}$ for all $N \in \NN$.
 
 	It follows from the Freiman--Ruzsa theorem (Thm.{} \ref{thm:FR}) that for each $N \in \NN$ there exists a proper symmetric GAP $Q_N$ of rank $\dd_N = O_K(1)$ as well as a finite set $F_N$ with cardinality $ h_N = \abs{F_N} =O_K(1)$ such that $\abs{Q_N} = O_K(\abs{A_N})$, $Q_N \subset 2A_N-2A_N$ and $A_N \subset F_N + Q_N$. Put $e := \max_N e_N$ and $h := \max_N h_N$. Adding to $F_N$ any $h - h_N$ points, we may freely assume that $h_N = h$ for all $N$. Similarly, we can construe $Q_N$ as a proper symmetric GAP of rank $e$, where the final $e-e_N$  side lengths are $0$ (and the corresponding steps are arbitrary). Thus, we may freely assume that $e_N = e$ for all $N$. Having performed the aforementioned reductions, we can write $Q_N$ in the form 	
\begin{equation}\label{eq:21:00}
	Q_N = \set{ n_1a_1^{(N)} + n_2 a_2^{(N)} + \dots + n_\dd a_\dd^{(N)}}{ \abs{n_i} \leq l_i^{(N)} \text{ for all } 1 \leq i \leq \dd}
\end{equation}
for some $a_1^{(N)},a_2^{(N)},\dots, a_\dd^{(N)} \in \RR^{\ee}$ and $l_1^{(N)},l_2^{(N)},\dots, l_\dd^{(N)} \in \NN_0$. 

While the cardinalities of the sets $F_N$ are uniformly bounded, these sets themselves \emph{a priori} do not need to be bounded as $N \to \infty$. This leads to technical complications, which we overcome in the following step.

\textbf{Claim 1.} \textit{There exist an integer $k = O_{K,\ee}(1)$ such that for each $N \in \NN$ there exists a finite set $F_N'$ with $\abs{F_N'} \leq h$, $F_N' \subset \BallR{\ee}{0}{O_K(1)}$ and $F_N + Q_N \subset F_N' + kQ_N$.}
\begin{proof}
	Let $N \in \NN$. 
	Since $Q_N \subset 2A_N - 2A_N \subset \BallR{\ee}{0}{4N}$, we may without loss of generality assume that 
	\[ F_N \subset A_N - Q_N \subset \BallR{\ee}{0}{5N}.\]
	Since $A$ is $1$-relatively dense,
	\[
		\BallR{\ee}{0}{N} \subset A_N + \bBallR{\ee}{0}{1} \subset F_N + Q_N + \bBallR{\ee}{0}{1}.
	\]
Hence, it follows from the union bound that the set 
\[ Q_N + \bBallR{\ee}{0}{1} \subset \BallR{\ee}{0}{4N+1} \subset \BallR{\ee}{0}{5N}
\] has volume at least $\lambda\bra{\BallR{\ee}{0}{N}}/h = \lambda\bra{\BallR{\ee}{0}{5N}}/(5^d h)$. By Corollary \ref{cor:Kemperman}, there exists $k_1 = O_{K,\ee}(1)$ (independent of $N$) and $b^{(N)} \in \RR^\ee$ such that 
	\[
	k_1(Q_N + \bBallR{\ee}{0}{1}) = k_1 Q_N + \bBallR{\ee}{0}{k_1} \supset \BallR{\ee}{0}{5N}  b^{(N)}.
	\] 
	Using the symmetry of $Q_N$, we may remove the shift by $b^{(N)}$:
	\[
	2k_1 Q_N + \BallR{\ee}{0}{2k_1}  \supset \BallR{\ee}{0}{5N}.
	\] 
	Hence, there exists a set $F_N' \subset \BallR{\ee}{0}{2k_1}$ with $\abs{F_N'} \leq \abs{F_N}$ such that
	\[
	F_N \subset 2k_1 Q_N + F_N'.
	\] 
	Letting $k = 2k_1 + 1 = O_{K,\ee}(1)$ we conclude that
	\[
	F_N + Q_N \subset F_N' + k Q_N. \qedhere
	\]
	\end{proof}

For $N \in \NN$, let us consider the subgroup $\Lambda_N$ of $\RR^{\ee+\dd}$ spanned by the $\dd$ vectors 
\begin{equation}\label{eq:19:00}
v_i^{(N)} = \bra{ a_i^{(N)}, \vec{e}_i/l_i^{(N)}}, \qquad (1 \leq i \leq \dd),
\end{equation}
where $\vec e_i$ denotes the $i$-th vector in the standard basis of $\RR^\dd$.
Note that these vectors are linearly independent, so they span a codimension $\ee$ subspace of $\RR^{\ee +\dd}$. Our next goal is to obtain a uniform bound on the shortest vector in $\Lambda_N$.

\textbf{Claim 2.} \textit{There exists $r > 0$ such that for each $N \in \NN$ the group $\Lambda_N$ is $r$-uniformly discrete.}
\begin{proof}

For the sake of contradiction, assume that for each $r > 0$ there exists $N$ such that $\Lambda_N$ is not $r$-uniformly discrete. Let $M \geq 1$ be an integer. By assumption, there exists an integer $N = N(M)$ such that $\Lambda_N$ contains a vector $u \in \Lambda_N \setminus \{0\}$ with $\norm{u}_{\infty} < 1/M$. Since $\Lambda_{N'}$ is discrete for each $N'$, we have $N \to \infty$ as $M \to \infty$. Since the vectors $ v_i^{(N)} $ are a basis of $\Lambda_N$, there exists a unique expansion
\begin{equation}\label{eq:21:10}
		u = \sum_{i=1}^\dd n_i v_i^{(N)},
\end{equation}
where $n_i \in \ZZ$ ($1 \leq i \leq e$). Let us consider the vector $c \in \RR^d$ given by
\begin{equation}\label{eq:21:10a}
		c := \sum_{i=1}^\dd n_i a_i^{(N)}.
\end{equation}
Bearing in mind \eqref{eq:19:00}, \eqref{eq:21:10} and \eqref{eq:21:10a}, we can rewrite the condition $\norm{u}_{\infty} < 1/M$ as
\begin{equation}\label{eq:21:11}
	\norm{ c }_{\infty} < 1/M, \text{ and } \abs{n_i} < l_i^{(N)}/M \text{ for } 1 \leq i \leq \dd.
\end{equation}
In particular, $m c \in Q_N \cap  \BallR{\ee}{0}{1}$ for all integers $m$ with $\abs{m} \leq M$. Note also that $c \neq 0$ since $Q_N$ is proper.

Let $B$ be a maximal $2$-separated subset of $A$, that is, a set $B \subset A$ maximal with respect to the property that $\norm{x-y}_{\infty} \geq 2$ for all $x,y \in B$ with $x \neq y$. Since $A$ is $1$-relatively dense and since for each $x \in A$ there exists $y \in B$ with $\norm{x-y}_{\infty} \leq 2$, we see that $B$ is $3$-relatively dense. In particular $D^{-}(B) > 0$. Put also $B_N := B \cap \BallR{\ee}{0}{N}$.

 Since $Q_N \subset 2A_N - 2A_N$ and $B_N \subset A_N$, we conclude that $m c + b \in 3A_N - 2A_N$ for all $-M \leq m \leq M$ and $b \in B_N$. Note also that all of these points $m c + b$ are pairwise distinct, and hence $\abs{3A_N - 2A_N} \geq (2M+1) \abs{B_N}$. On the other hand, it follows from the Pl\"{u}nnecke inequality (Thm.\ \ref{thm:Plunnecke}) that $\abs{3A_N - 2A_N} \leq K^5 \abs{A_N}$  and consequently
 \begin{equation}\label{eq:21:11a}
M \leq K^5 \abs{A_N}/\abs{B_N}.
\end{equation}
The expression on the right-hand side of \eqref{eq:21:11a} is bounded as $M \to \infty$; in fact, bearing in mind that that $N = N(M) \to \infty$ as $N \to \infty$ we have
\[
	\limsup_{M \to \infty} K^5 \abs{A_N}/\abs{B_N} \leq K^5 D^{+}(A)/D^{-}(B).
\]
Thus, for sufficiently large $M$, \eqref{eq:21:11a} yields a contradiction.
\end{proof}

We next show that the sets $A_N$ are contained in shift-cut-and-project sets coming from $\Lambda_N$.
Recall that $F_N'$ and $k$ are defined in Claim 1 above.

\textbf{Claim 3.} \textit{For each $N \in \NN$ it holds that}
\begin{equation}\label{eq:21:12}
A_N \subset F_N' + \pi_1\bra{\Lambda_N \cap (\RR^{\ee} \times \bBallR{\dd}{0}{k})} .
\end{equation}

\begin{proof}
	Pick any $N$ and any $x \in kQ_N$. Then $x$ can be written as
	\[
		x = \sum_{i=1}^\dd n_i a_i^{(N)},
	\]
	where $\abs{n_i} \leq kl_i^{(N)}$ for all $1 \leq i \leq \dd$. Hence, $x$ is the projection onto the first coordinate of the point
	\[
		\bra{ x, \ \sum_{i=1}^\dd (n_i/l_i^{(N)})\vec e_i} \in \Lambda_N \cap \bra{\RR^\ee \times \bBallR{\dd}{0}{k}}.
	\]
	It follows that \eqref{eq:21:12} holds:
	\[
		A_N \subset F_N' + kQ_N' \subset F_N' + \pi_1\bra{\Lambda_N \cap \bra{ \RR^\ee \times \bBallR{\dd}{0}{k}}}. 
		\qedhere
	\]
\end{proof}

	Passing to a subsequence and using Claim 2, we may assume that the Kuratowski limit of $\Lambda_N$ exists (cf.{} discussion in Section \ref{ssec:SetLimit}). More precisely, we assume that there exists $\Lambda \subset \RR^{d+e}$ and an infinite set $\cN = \{N_1 < N_2 < N_3 < \dots \} \subset \NN$ with 
	\[
		\klim_{\cN \ni N \to \infty} \Lambda_N := \klim_{i \to \infty} \Lambda_{N_i} = \Lambda.
	\]
Since the sets $F_N'$ are uniformly bounded and have uniformly bounded cardinalities, 	replacing $\cN$ with a smaller set if necessary, we may assume that $F_N' = \{b_1^{(N)},b_2^{(N)},\dots,b_h^{(N)}\}$ where for each $1 \leq i \leq h$, $b_i^{(N)}$ converges to some point $b_i \in \RR^{\ee}$ as  $\cN \ni N \to \infty$. We put $F' = \{b_1,b_2,\dots,b_h\}$ and note that $\abs{F'} \leq h$ (the inequality may be strict). We are now ready to express $A$ as a subset of a shift-cut-and-project set coming from $\Lambda$.

\textbf{Claim 4.} \textit{The set $\Lambda$ is a discrete subgroup of $\RR^{\ee} \times \RR^{\dd}$ and}
\begin{equation}\label{eq:21:13}
A \subset  F' + \pi_1\bra{\Lambda \cap \bra{ \RR^\ee \times \bBallR{\dd}{0}{k}}}.
\end{equation}
\begin{proof}
Recall from Lemma \ref{lem:klim-props} that the property of being $r$-uniformly discrete is preserved under Kuratowski limits, and the same applies to the property of being a group. Together with Claim 2, this implies that $\Lambda$ is an $r$-discrete group.

To show \eqref{eq:21:13}, take any point $x \in A$. Let $N_0 > \norm{x}_\infty$ be an integer, so that $x \in A_N$ for all $N \geq N_0$. If follows from Claim 3 that for all $N \in \cN$ with $N \geq N_0$ there exist $y_N \in F_N'$ and $z_N \in \Lambda_N \cap \bra{ \RR^\ee \times \bBallR{\dd}{0}{k} }$ such that $x = y_N + \pi_1(z_N)$. 
 By Claim 1, $y_N$ are bounded as $N \to \infty$. Since $\norm{\pi_1(z_N)}_\infty \leq  \norm{x}_\infty+\norm{y_N}_\infty$ and $\norm{\pi_2(z_N)}_{\infty} \leq k$, also $z_N$ are bounded as $N \to \infty$.
Let $y$ and $z$ be accumulation points of the sequences $(y_N)_{N \in \cN}$ and $(z_N)_{N \in \cN}$ respectively. Directly by the definitions of $F'$ and $\Lambda$, we have $y \in F'$ and $z \in \Lambda$. It also follows from basic properties of limits that $z \in \RR^\ee \times \bBallR{\dd}{0}{k}$ and $x = y + \pi_1(z)$. Since $x \in A$ was arbitrary and $\bBallR{\dd}{0}{k}$, \eqref{eq:21:13} follows.
\end{proof}

Formula \eqref{eq:21:13} implies immediately that
\begin{equation}
A \subset  F' + \pi_1\bra{\Lambda \cap \bra{ \RR^\ee \times \BallR{\dd}{0}{k+1}}},
\end{equation}
This says exactly that $A$ is contained in a shift-cut-and-project set, which finishes the argument (cf.{} Proposition \ref{lem:cut-and-project-WLOG}). 
\end{proof} 
\section{Quantitative estimates}\label{sec:Epilogue}

The search for quantitative variants of the Freiman--Ruzsa theorem has been a subject of vigorous investigation in additive combinatorics. To discuss these developments it is convenient to replace the GAP $P$ appearing in Theorem \ref{thm:Freiman-Ruzsa} with a set of the form $F+Q$ where $F$ is a finite set and $Q$ is a GAP; we will be interested in bounds on the cardinality of $F$ and the rank and the size of $Q$. The rationale behind this reformulation is that one cannot, in general, hope that any set $A \subset \ZZ$ with doubling $K$ would be contained in an GAP of rank significantly less than $K$, as can be seen by considering a dissociated set of size $K$ (for definition of a dissociated set, see e.g. \cite[Def.\ 4.32]{TaoVu-book}). On the other hand, one can hope that $A$ should be contained in $F+Q$ where the size of $F$ and rank of $Q$ can be accurately bounded. Note also that for any finite set $F$ and any GAP $Q$ of rank $d$, the set $F+Q$ is contained in a GAP of rank $\abs{F}+d$ and size $2^{\abs{F}}\norm{Q}$.

\begin{theorem}[Freiman--Ruzsa, covering version]
	For any $K \geq 2$ there exists $d = d(K)$ such that the following holds.
	Let $A$ be a subset of an abelian torsion-free group $Z$ such that $\abs{2A} \leq K \abs{A}$ for some $K \geq 2$. Then there exists a GAP $Q$ of rank $d$ and a finite set $F$ with $\abs{F} \leq 2^d$ such that $A \subset F+Q$ and $\abs{F+Q} \leq 2^d \abs{A}$.
\end{theorem}

\newcommand{\D}{d_{\mathrm{FR}}}
Let $\D(K)$ denote the least value of $d$ for which the above the theorem holds for a given value of $K \geq 2$. The original argument by Freiman \cite{Freiman-book} did not provide useful estimates on $\D(K)$, and the first quantitative bounds come from the work of Ruzsa \cite{Ruzsa-1994}.  The next significant improvement, $\D(K) \leq K^{2+o(1)}$, is due to Chang \cite{Chang-2002}. (Here, $o(1)$ denotes a quantity that tends to $0$ as $K \to \infty$.) The exponent was later reduced to $\D(K) \leq K^{7/4+o(1)}$ by Sanders \cite{Sanders-2008}. In a breakthrough paper, Schoen \cite{Schoen-2011} obtained the first sub-polynomial bound $\D(K) \leq \exp \big( {O( \sqrt{\log(K)})} \big)$. The current record belongs to Sanders, who showed \cite{Sanders-2012} that $\D(K) \leq \log^{3+o(1)}(2K)$ (the factor $2$ ensures that $\log(2K) > 1$). The polynomial Freiman--Ruzsa conjecture, which is one of the major open problems in additive combinatorics, asserts that $\D(K) = O(\log K)$; if true, this is optimal (up to improvement of the constant implicit in the $O(\cdot)$ notation).
We also have the analogue of Theorem \ref{thm:FR}.
\begin{theorem}\label{thm:FR-2} 
		For any $K \geq 2$ there exists $d = d(K)$ such that the following holds.
		 Let $A,B$ be subsets of a torsion-free abelian group $Z$, and suppose that $\abs{A} = \abs{B} = n$ and $\abs{A+B} \leq Kn$ for some $n > 0$. Then there exists a proper symmetric GAP $Q$ of rank at most $d$ and cardinality at most $2^d n$ and a finite set $F$ of cardinality at most $2^d$ such that $Q \subset 2A-2A$ and $A \subset F+Q$. 
\end{theorem}

		 Let $\D'(K)$ denote the least value of $d$ for which Theorem \ref{thm:FR-2} holds. We have a quantitative estimate, analogous to the aforementioned result in \cite{Sanders-2012}.   The following argument was pointed out to the author by Tom Sanders.

\begin{proposition}\label{prop:FR-2} 
	For each $K \geq 2$ we have $\D'(K) = \log^{6+o(1)}(2K)$.
\end{proposition}
\begin{proof}
	It follows from the variant of Bogolyubov--Ruzsa lemma in \cite{Sanders-2012} that $2A - 2A$ contains a proper symmetric GAP $Q$ of rank at most $d$ and size at least $2^{-d}\abs{A}$, where $d = \log^{6+o(1)} K$. Let $R$ be the $d$-dimensional progression obtained by halving all sidelengths of $Q$ (we may assume without loss of generality that they are all even), and let $F$ be a maximal subset of $A$ such that all of the sets $x + R$ ($x \in F$) are disjoint. Then $A \subset F+R-R \subset F+Q$ and $\abs{F} \leq 2^{O(d)}$. 
\end{proof}
\begin{remark}
If it was not for the condition $Q \subset 2A-2A$ in Theorem \ref{thm:FR-2}, we could infer Proposition \ref{prop:FR-2} directly from the bound $\D(K) \leq \log^{3+o(1)}(2K)$  using the same argument as in Section \ref{sec:Additive}.
\end{remark}

\begin{corollary}\label{cor:final}
	Let $A \subset \RR^{\ee}$ be a relatively dense set and suppose that $A - A$ is uniformly discrete. Then $A \subset M+F$ where $M$ is a cut-and-project set with at most $ \D'\bra{2^d {D^+(A-A)}/{D^{-}(A)}+1}$ internal dimensions and $F$ is a finite set.
\end{corollary}
\begin{proof}
We follow the proof of Theorem \ref{thm:main}, keeping track of bounds and using Theorem \ref{thm:FR-2} in place of Theorem \ref{thm:FR}.
Restricting our attention to sufficiently large $N$, we may assume that $K \leq 2^d D^{+}(A-A)/D^{-}(A)+1$. By Theorem \ref{thm:FR-2}, we may take $\dd \leq \D'(K)$. It remains to recall that the number of internal dimensions of the cut-and-project set produced in the argument does not exceed $\dd$.
\end{proof}

Theorem \ref{thm:main-quant} follows directly from Corollary \ref{cor:final} and Proposition \ref{prop:FR-2}. In analogy with the polynomial Freiman--Ruzsa conjecture, one can conjecture that $\D'(K) = O(\log(K))$. If so, in Theorem \ref{thm:main-quant} it is enough to use $O\bra{d + \log \bra{{D^+(A-A)}/{D^{-}(A)}}}$ internal dimensions.

\bibliographystyle{alphaabbr2}

\bibliography{bibliography}

\end{document}